\documentclass[11pt,a4paper]{article}
\usepackage[utf8]{inputenc}

\usepackage{amssymb,amscd,amsfonts,amsbsy,amsmath,amsthm,mathrsfs}
\usepackage{dsfont}
\usepackage{enumerate,color}
               \usepackage{pdfsync}   

\newtheorem{proposition}{Proposition}[section]
\newtheorem{theorem}[proposition]{Theorem}
\newtheorem{lemma}[proposition]{Lemma}
\newtheorem{corollary}[proposition]{Corollary}

\theoremstyle{definition}

\theoremstyle{remark}
\newtheorem{remark}[proposition]{Remark}

\numberwithin{equation}{section}

\newcommand{\dd}{{\rm\,d}}
\newcommand{\dx}{{\rm\,d}x}
\newcommand{\ds}{{\rm\,d}s}
\newcommand{\dsigma}{{\rm\,d}\sigma}
\newcommand{\R}{\mathbb{R}}

\newcommand{\Y}{\mathscr{Y}}

\newcommand{\PM}{\mathcal{P}\!\mathcal{M}}

\newcommand{\III}[1]{{\left\vert\kern-0.25ex\left\vert\kern-0.25ex\left\vert #1 
    \right\vert\kern-0.25ex\right\vert\kern-0.25ex\right\vert}}

\DeclareMathOperator{\esssup}{ess\,sup}

\begin{document}
     \baselineskip=13pt
\title{
Large global solutions of the parabolic-parabolic \\ Keller--Segel system
 in higher dimensions} 
\author{Piotr Biler$^1$, Alexandre Boritchev$^2$ and Lorenzo Brandolese$^2$\\ 
\small{ $^1$ Instytut Matematyczny, Uniwersytet Wroc{\l}awski,}\\ 
\small{ pl. Grunwaldzki 2, 50--384 Wroc{\l}aw, POLAND}\\ 
\small{ $^2$   Université Claude Bernard Lyon 1,}\\
\small{CNRS UMR 5208, Institut Camille Jordan,}\\
\small{F-69622 Villeurbanne, FRANCE }\\ 
}

\date{\today}
\maketitle

\begin{abstract}
We study the global existence of the parabolic-parabolic Keller--Segel system
in $\R^d$, $d\ge2$. We prove that initial data of arbitrary size give rise to global solutions provided the diffusion parameter $\tau$ is large enough in the equation for the chemoattractant. This fact was observed before in the two-dimensional case by Biler, Guerra \& Karch (2015) and Corrias, Escobedo \& Matos (2014). Our analysis improves earlier results 
and extends them to any dimension $d\ge3$.
Our size conditions on the initial data for the global existence of solutions seem to be optimal, up to a logarithmic factor in $\tau$, when $\tau\gg1$: we illustrate this fact by introducing two toy models, both consisting of systems of two parabolic equations, obtained after a slight modification of the nonlinearity of the usual  Keller--Segel system. For these toy models,  we establish in a companion paper \cite{BBB2} 
finite time blowup for a class of large solutions.
\end{abstract}

\section{Introduction}

This paper is concerned with the simplest doubly parabolic Keller--Segel system 
\begin{equation}
\left\{
\begin{aligned}
&u_t =\Delta u-\nabla\cdot(u\nabla \varphi),\\
&\tau \varphi_t=\Delta \varphi+u,\\
&u(0)=u_0,\ \  \varphi(0)=\varphi_0,
\end{aligned}
\right.
\qquad x\in\R^d,\ t>0,
\tag{PP}
\end{equation}
where $\tau>0$. 
There are two topics of interest: global-in-time existence of solutions {\em versus} finite time blowup of solutions. 
The biological motivations to consider the system {\rm(PP)} are related to the model of chemotaxis, i.e. motion of microorganisms of density $u=u(x,t)\ge 0$ which are subject to diffusion described by the Laplacian $\Delta u$ and drift along the gradient of the chemoattractant density $\nabla\varphi$ --- a chemical secreted by themselves --- playing a role of the  information carrier in some sense. 
The most interesting situation in applications is that with either small coefficient $0<\tau\ll 1$, or $\tau=0$. This means that diffusion for the chemoattractant is much faster than for cells, or even instantaneous if $\tau=0$, which leads to the parabolic-elliptic Keller--Segel system:
\begin{equation}
\left\{
\begin{aligned}
&u_t =\Delta u-\nabla\cdot(u\nabla \varphi),\\
&\Delta \varphi+u=0,\\
&u(0)=u_0.  
\end{aligned}
\right.
\qquad x\in\R^d,\ t>0,
\tag{PE}
\end{equation}
Results on the continuity of solutions of those systems with respect to the parameter $\tau$, in the limit $\tau\searrow 0$, are available for solutions with sufficiently small initial data, see e.g. \cite{Rac,BB2,Lem}. 
Note that the case $\tau=0$ is also motivated by applications in astrophysics when massive particles in a cloud of matter (say, a star, or a nebula) are attracted through the gravitational potential created by themselves, responding  instantaneously to the configuration of particles that evolves in time, see \cite{CSR}. 
In this model, that is the parabolic-elliptic  Keller--Segel system (PE), initial data which are large in a suitable sense lead to the phenomenon of finite time blowup of solutions. This is an interesting property from the viewpoint of the mathematical modelization in biology, which also has a striking purely mathematical meaning summarized briefly by saying that transport prevails over diffusion. There are many results on the blowup for the parabolic-elliptic Keller--Segel model, beginning with the pioneering observations in \cite{JL}, through numerous papers of the Japanese school by T. Nagai, Y. Naito, T. Senba, T. Suzuki, up to newest results on radially symmetric solutions in \cite{N2021,SW}. 
The two questions: {\em ``What are sufficient conditions on the initial data for the existence of global-in-time solutions~?"} and {\em ``What are sufficient conditions on the initial data for a finite time blowup~?"} are closely related and lead to results of (partial) dichotomy; those conditions are, in a sense,  complementary, see e.g. \cite{B-book} for a review of recent results.

The analogous questions for the doubly parabolic system {\rm(PP)} lead to a quite satisfactory theory of the existence of local and global-in-time solutions (for references, see the review in \cite{B-book}) culminating in, e.g. \cite{CCE,Lem}, ad also \cite{Tak}. 
On the other hand, the blowup is not so well understood since many of standard by now methods for single parabolic equations (see \cite{QS}) and parabolic-elliptic Keller--Segel system fail for the doubly parabolic Keller--Segel model, see however \cite{CEM} for concentration phenomena. 
For newest results on blowups  we refer to  \cite{Wi1,Wi2} where  the radially symmetric problem is considered in a ball, and in \cite{Wi3} in the whole space.  A supplementary information on $L^1$ solutions is derived from entropy functionals and other specific properties of those drift-diffusion systems.
The problem of proving blowup of solutions without specific regularity and symmetry properties for some parabolic systems 
requires perhaps new methods. 

\subsection*{Overview of the results}  

The purpose of this paper is to provide size conditions on the initial data $u_0$ and $\varphi_0$ in appropriate scale invariant norms, and explicitly depending on the parameter $\tau$, guaranteeing that solutions to (PP) are global-in-time. 
Of special mathematical interest
are the two asymptotic regimes $0<\tau\ll1$ and $\tau\gg1$.

A larger $\tau$ leads to a stronger dissipation and one should expect in this case the global-in-time existence of solutions for larger initial data $u_0$.
Results in this directions were obtained in the two-dimensional case in \cite{BCD,BGK,CEM}. Our goal is twofold: first we would like to extend these results to the case $d\ge3$. Second, for the two dimensional case, we would like to improve the admissible size conditions on the initial data available in the literature.

For instance, for large $\tau$, in any dimension $d\ge2$ we will be able to construct global solutions
under size conditions in pseudomeasure norms (see next section for the definition)
of the form
\begin{equation}
\label{in:phi1}
\|u_0\|_{\PM^{d-2}}\lesssim\tau/(\ln\tau)^3
\quad \text{($\tau\gg 1$)},
\end{equation}
for an initial concentration $\varphi_0$ satisfying appropriate size conditions
in another pseudomeasure space.
In particular,  the initial density $u_0$ can be taken arbitrarily large when
$\tau\gg1$.
The closest results available  in the literature are those of~\cite{BGK, CEM} (just in the case $d=2$).
But in these references, the size condition on $u_0$ was considerably more restrictive (of the form $O(\tau^{1/2-\epsilon}$ with $\epsilon>0$) for the norm of $u_0$, so our results improve earlier studies by a factor better than $\sqrt \tau$.

\medskip
In the other asymptotic regime, $0<\tau\ll 1$, the situation is different.
We can construct global solutions under size conditions of the form
\begin{equation}
\label{in:phi0}
\begin{cases}
\|u_0\|_{\PM^{d-2}}\lesssim 1\\
\|\nabla\varphi_0\|_{\PM^{d-1}}\lesssim \max\{1,\frac{1}{\sqrt\tau}\}
\end{cases}
\quad(\tau>0),
\end{equation}
for $d\ge3$, and even for $d=2$ with a logarithmic correction in $\tau$.
In particular, the initial concentration $\varphi_0$ can be taken arbitrarily large when $\tau$ is small.
Of course, our analysis  of the global existence to (PP) in the case $0<\tau\ll 1$ is
related to the convergence problem of solutions to (PP) to those to (PE), as $\tau\searrow0$, addressed in~\cite{Rac,BB2,Lem}.

In order to illustrate the importance of the choice of the function spaces  for constructing global solutions, we address the global existence problem also under the more general assumption that the initial data $u_0$ belong to appropriate homogeneous Besov spaces. For illustration purpose, it is sufficient to limit ourselves to the 
model case 
$\varphi_0=0$, $u_0\not=0$.
The Besov spaces that we use are larger than the corresponding pseudomeasure spaces,
but they are still not optimal (in the sense of the inclusion). For the physically relevant case $u_0\ge0$, the well-posedness in the largest possible
scale invariant function space 
was successfully addressed in~\cite{Lem}, using a space of Morrey type.

But the price to pay when dealing with rougher spaces is important: 
the weaker the norm $\|\cdot\|$, the more stringent asymptotically as $\tau\to+\infty$ is the size condition on $\|u_0\|$ that one has to prescribe.
In the setting of  Besov spaces our best size condition is
$\|u_0\|_{\dot B^{-(2-d/p)}_{p,\infty}}\lesssim C_p\sqrt \tau$ (with $p\in(d,2d)$),
which is much worse than~\eqref{in:phi1} when $\tau\gg 1$.
In the optimal Morrey space $\dot M^{d/2}$, the size condition turns out to be even more stringent for large~$\tau$: indeed, the main result of~\cite{Lem} requires 
$\|u_0\|_{\dot M^{d/2}}\lesssim 1$.

Pseudomeasure spaces thus seem to be a good compromise. 
They allow to better take advantage of the role of the parameter $\tau$. Yet, the space $\PM^{d-2}$ is large enough to encompass homogeneous distributions of degree $-2$ that give rise to self-similar solutions.
In fact, our original motivation to study {\rm(PP)} with large $\tau$ was a characterization of the initial data $u_0=M\delta_0$ in $d=2$ leading to the existence (and sometimes nonuniqueness) of self-similar solutions to {\rm(PP)} in \cite{BCD}, as well as \cite{BGK} again in $d=2$; see Corollary \ref{sss}.

Our main results are Theorem~\ref{th:PM}, where
we construct global solutions in pseudomeasure spaces, and  
Theorem~\ref{BesT1}. The latter deals with Besov spaces and can be viewed as
the counterpart, for $\tau>0$, of an analogous result established in~\cite{I}
for the parabolic-elliptic system (PE).

All our results remain valid for sign-changing solutions, as we do not need to put the conditions 
$u_0\ge0$, or $\varphi_0\ge0$. They remain valid also when an additional damping is put in the equation
for $\varphi$, i.e., when the second equation in (PP) is rewritten as
$\tau\varphi_t=\Delta \varphi+u-\alpha\varphi$, with $\alpha\ge0$. Indeed,
the contribution of this damping term can simply be dropped in all our estimates. In fact, taking $\alpha>0$ could possibly lead to some improvements on our size conditions.



\subsection*{Further developments}

In a companion paper \cite{BBB2} we deal with parabolic systems {\rm (TM)} and
 {\rm (TM')} below, in which the nonlinear term $\nabla\cdot(u\nabla \varphi)$
is replaced by a nonlinearity of the same order and scaling, but without drift structure. 
The first model is
\begin{equation}
\left\{
\begin{aligned}
&u_t =\Delta u-u \Delta \varphi,\\
&\tau \varphi_t=\Delta \varphi+u,\\
&u(0)=u_0,\ \  \varphi(0)=\varphi_0,
\end{aligned}
\right.
\qquad x\in\R^d,\ t>0,
\tag{TM}
\end{equation}
The second model is
\begin{equation}
\left\{
\begin{aligned}
&u_t =\Delta u+(\Delta\varphi)^2,\\
&\tau \varphi_t=\Delta \varphi+u,\\
&u(0)=u_0,\ \  \varphi(0)=\varphi_0,
\end{aligned}
\right.
\qquad x\in\R^d,\ t>0,
\tag{TM'}
\end{equation}
Both models degenerate into the quadratic nonlinear heat equation in the parabolic-elliptic limit 
$\tau=0$, when one removes the initial condition on $\varphi$.
So they are related to the Cauchy problem for the quadratic heat equation
\begin{equation}\tag{NLH}
\left\{\begin{aligned}
&u_t =\Delta u+u^2,\quad\\
&u(0)=u_0,\quad
\end{aligned}
\quad \qquad \qquad x\in\R^d,\ t>0,
\right.
\end{equation}
in a way similar to that as {\rm(PP)} relates to (PE).
Moreover, under the compatibility condition $\Delta\varphi_0+u_0=0$,
the steady states of (TM) and (TM') agree with those of (NLH).

For both toy models, the existence theory that we present in this paper for (PP)  applies with only slight changes. 
The interesting feature of (TM) and (TM') is that the blowup problem  is much better understood than for (PP):
in the companion paper~\cite{BBB2} we will provide 
some explicit blowup criteria for initial data of large size, not necessarily radial, using different methods.
The blowup analysis for (TM) and (TM') illustrates also the nearly optimal character (both in terms of the regularity  and the size of the initial data)
of our global existence results. 
In particular, it suggests that the size condition~\eqref{in:phi1} on $u_0$ for the global existence of (PP) is optimal for large $\tau$, up to a logarithmic factor.

\subsection*{Notation}

In this paper we adopt the following notation and conventions. The expression $A\lesssim B$, where $A$ and $B$ may depend on several parameters, means that there exists a constant $c>0$, dependent only on the space dimension, such that $A\le c\,B$. When both $A\lesssim B$ and $B\lesssim A$
we will write $A\approx B$. 
The notation $a\wedge b$ stands for $\min\{a,b\}$,

The space of linear and continuous operators from a Banach space $Y$ to another Banach space~$Z$ is denoted 
by $\mathscr{L}(Y,Z)$. This is a Banach space when it is endowed with the natural norm $\|T\|_{\mathscr{L}(Y,Z)}=\sup_{\|u\|_Y=1}\|Tu\|_Z$. When $Z=Y$ we simply write $\mathscr{L}(Y)$ instead of $\mathscr{L}(Y,Z)$.

For a bilinear and continuous operator $B$ from $Y\times Y$ to $Y$ we denote
by $|\!|\!|B|\!|\!|$ its bilinear operator norm, i.e., 
$|\!|\!|B|\!|\!|=\sup \|B(u,v)\|_Y$, where the supremum is taken over all $u,v\in Y$ such that
$\|u\|_Y=\|v\|_Y=1$.

For a function $f\in L^1(\R^d)$, 
the definition of the Fourier transform that we use is 
$\widehat f(\xi)=\int f(x)\exp(-i\xi\cdot x)\dd x$. This definition is extended
to $\mathscr{S}'(\R^n)$, the space of tempered distributions, in the usual way.
The space of general distributions is denoted $\mathscr{D}'(\R^d)$.

We denote by $M^s_q(\R^d)$ the homogeneous Morrey spaces. For $1\le q\le s<\infty$ these are normed by
\[
\|u\|_{M^s_q}=\Bigl(\sup_{R>0,x\in\R^d} R^{d(q/s-1)}\int_{|y-x|<R}|u(y)|^q\dd y\Bigr)^{1/q}.
\]
When $q=1$, we write simply $M^s$ instead of $M^{s}_1$.

In this paper we will deal with mild solutions. 
These are solutions of the integral formulation of (PP). 
The exact meaning of the integral must be understood in the specific functional setting. For example, in the setting of Theorem~\ref{th:PM},
by definition, a mild solution to (PP) is a map $u\in L^1_{\rm loc}(0,\infty;\mathscr{S}'(\R^d))$ 
satisfying~equation \eqref{Fsol'} below for all $t\ge0$ and a.e. 
$\xi\in\R^d$.
This definition allows us to disregard $\varphi$, so that $u$ plays a predominant role in (PP).

\section{Study of {\rm (PP)} in pseudomeasure spaces}

Consider the parabolic-parabolic Keller--Segel system {\rm (PP)}
with $\tau>0$. 
Let $0\le a<d$. We introduce the pseudomeasure space
\begin{equation}
\label{PMa}
\PM^a=
\{f\in \mathscr{S}'(\R^d)\colon \|f\|_{\PM^a}=\esssup_{\xi\in\R^d}|\xi|^a|\widehat f(\xi)|<\infty\},
\end{equation}
where $\widehat f$ denotes the Fourier transform for the tempered distribution $f$.
The idea of using pseudomeasure spaces $\PM^a$ for studying the Keller--Segel system goes back to~\cite{BCGK}, and classically pseudomeasures with $a=0$ have been considered in harmonic analysis. 
The space $\PM^{d-2}$ 
plays a predominant role for $u$, because of the scale invariance $u(x,t) \mapsto \lambda^2u(\lambda x,\lambda^2 t)$ of the equation, and the fact that 
\[
\|\lambda^2u_0(\lambda\cdot)\|_{\PM^{d-2}}
=\|u_0\|_{\PM^{d-2}}
\]
for any $\lambda>0$.
For this reason, we will prescribe a size condition using this norm.

On the other hand, the scaling for $\varphi$ is $\varphi(x,t)\mapsto \varphi(\lambda x,\lambda^2 t)$: as $\Delta\varphi$ scales as $u$, one could be tempted to prescribe a size condition on $\|\varphi_0\|_{\PM^d}$. But pseudomeasure spaces $\PM^a$ are not well defined when $a\ge d$, as $|\xi|^{-a}$ does no longer belong to $L^1_{\rm loc}(\R^d)$. We can conveniently circumvent this obstruction by assuming 
$\varphi_0\in\mathscr{S}'(\R^d)$ and
 putting a size condition on $\|\nabla\varphi_0\|_{\PM^{d-1}}$.
In fact, only $\nabla \varphi$ plays a role in the equation for $u$.

We will construct our solutions in the space which, again, is scale-invariant.
\begin{equation}
\Y_a=\{u\in L^\infty_{\rm loc}(0,\infty;\mathscr{S}'(\R^d))\colon
\|u\|_{\Y_a}=\esssup_{t>0,\,\xi\in\R^d}t^{1+(a-d)/2}|\xi|^a|\widehat u(\xi,t)|<\infty\}.
\end{equation}
When $a=d-2$, the space $\Y_{d-2}$ agrees with the space 
\[
\mathcal{X}=L^\infty(0,\infty;\PM^{d-2}),
\]
already used in \cite[Theorem 2.1]{BCGK} to establish a global existence  result
for the parabolic-elliptic Keller--Segel system for $d\ge4$
and small initial data in $\PM^{d-2}$.
When $a\not=d-2$, our space $\Y_a$ is slightly larger than the space
\[
\mathcal{Y}_a:=\Y_a\cap\mathcal{X}
\]
 considered in \cite[Section 4]{BCGK}  $(d\ge3)$
or in \cite{Rac} $(d=2)$.
Such a slight change in the choice of the functional setting for the relevant bilinear estimates will have, however, important consequences.
Constructing the solution in the larger space $\Y_a$, instead of $\mathcal{Y}_a$, implies that we do not have to care about the estimates of the $\mathcal{X}$-norm: this represents a crucial advantage when $\tau$ is large, because the bilinear operator norm 
in $\Y_a$ does go to zero as $\tau\to+\infty$, while that of the bilinear operator norm
in $\mathcal{Y}_a$ does not. Moreover, the solution constructed in $\Y_a$ {\it a~fortiori} also belongs to
$\mathcal{X}$, so, in fact, no information is lost at the end.

\begin{theorem}
\label{th:PM}
\begin{itemize}
\item[(i)]
Let $d\ge2$.
There exist two constants  $\kappa_d$ 
and $\tilde\kappa_d>0$
such that,
if $\tau>0$ and $u_0\in \PM^{d-2}(\R^d)$, 
$\varphi_0\in \mathscr{S}'(\R^d)$ 
satisfy one of the following size conditions
\begin{subequations}
\begin{equation}
\begin{aligned}
\label{small-assu0}
\|u_0\|_{\PM^{d-2}} &<\kappa_d, \\
\sqrt\tau\,\|\nabla\varphi_0\|_{\PM^{d-1}} &<\tilde\kappa_d
\\
\end{aligned}
\qquad 
(0<\tau\le 1)
\end{equation}
(when $d=2$, the second condition in~\eqref{small-assu0} has to be replaced by the 
more stringent condition
\[
|\ln\textstyle\frac{\tau}{{\rm e}}|\sqrt{\tau}\,\|\nabla\varphi_0\|_{\PM^{d-1}} 
<\tilde\kappa_2, \qquad \text{($d=2$,\quad$0<\tau\le1$}
))\]
or otherwise,
\begin{equation}
\begin{aligned}
\label{small-assu}
&\|u_0\|_{\PM^{d-2}} <\kappa_d\, b^3\tau^{1-b}, \\
&\|\nabla\varphi_0\|_{\PM^{d-1}} <\tilde \kappa_d \, b^2 \\
\end{aligned}
\qquad 
(\text{for $\tau\ge 1$ and some $0<b\le1$}),
\end{equation}
then  {\rm(PP)} possesses a global mild solution~$u\in \mathcal{X}$.
\end{subequations}
\item[(ii)]
There exists $a\in [d-2,d)$ and $R>0$ such that
such a solution belongs to 
$\{v\in \Y_{a}\colon \|v\|_{\Y_{a}}<R\}$, and is uniquely
defined in this ball.
\end{itemize}
\end{theorem}

\begin{remark}
\label{rem:di}
In the second item, the parameter $a$ depends on $d$, $\tau$ and $b$. For example, 
when $0<\tau\le 1$ and $d\ge2$ one can take $a=d-\frac43$ (or $a=d-2$ when $d\ge4$).
When $\tau \ge1$ and $d\ge3$ one can take $a=d-\frac43 b$ (or $a=d-2b$ when $d\ge4$).
The choice $a=d-\frac43 b$ is valid, in fact, also
for $d=2$ when $0<b\le \frac12$. 
(A modification would be needed when $d=2$,  
$\tau\ge1$  and $\frac12<b\le1$.  But this case is not interesting: if assumption~\eqref{small-assu} is satisfied for such a $b$, then it is satisfied also with $b/2$ instead of~$b$ with a different choice of the constants $\kappa_2$ and
$\tilde \kappa_2$,
and so one can reduce to the previous case).
The radius can be be chosen as follows: $R=4\kappa_d$ when~\eqref{small-assu0} holds
and $R=4\kappa_d\, b^3\tau^{1-b}$ when~\eqref{small-assu} holds.
\end{remark}

In the above theorem, the parameter $b\in(0,1]$ can be tuned as we like: we can choose a fixed $b$ or otherwise a function $b=b(d,\tau,\varphi_0, u_0)$.
When $\tau\gg 1$ an interesting choice is $b=3/\ln\tau$.
Indeed, this is the choice allowing the weakest possible size condition for $u_0$
when $\tau$ is large.
For example, in the model case $\varphi_0=0$ we get the following result:

\begin{corollary}
\label{cor:tl}
Let $d\ge2$, $u_0\in \PM^{d-2}$ and $\varphi_0=0$. If $\tau\ge {\rm e}^3$,
then (PP) possesses a global solution.
under the smallness condition
\begin{equation}
\label{small-assuc}
\|u_0\|_{\PM^{d-2}} < 3^3\kappa_d\,\tau/({\rm e}\ln\tau)^3.
\end{equation}
Such a solution belongs to $\mathcal{X}\cap\Y_{d-4/\ln\tau}$
and is unique in a ball of $\Y_{d-4/\ln\tau}$
centered at the origin, with radius $0<r\lesssim \tau/(\ln\tau)^3$
.
\end{corollary}

Corollary~\ref{cor:tl} should be compared with the  results in~\cite{BGK} and \cite{CEM} for $d=2$: therein, the authors assumed $u_0$ to belong to the space of finite Radon measures and in $L^1(\R^2)$, respectively. Then they proved that
solutions to {\rm(PP)} are global provided the size of the initial data in such spaces does not exceed $c_\epsilon\max(1,\tau^{{1/2}-\epsilon})$ for each $\epsilon>0$ and some constant $c_\epsilon>0$.
On one hand, the space $\PM^0(\R^2)$ that we consider is larger; on the other hand our size conditions~\eqref{small-assu}, or~\eqref{small-assuc}, are weaker than those in~\cite{BGK}, \cite{CEM}.

%

When $d=2$ and $u_0(x)=M\delta_0$, or when $d\ge3$ and $u_0=M|x|^{-2}$, then $u_0\in \PM^{d-2}$ and $\|u_0\|_{\PM^{d-2}}\lesssim M$.
Corollary~\ref{th:PM} can be applied to establish the existence of self-similar solutions with $M={\mathcal O}(\tau/(\ln\tau)^3)$ as $\tau\to\infty$.
In particular, an immediate consequence of Corollary~\ref{cor:tl} is the following:

\begin{corollary}\label{sss} 
Let $d\ge2$.
For each $M \in \R$ and the initial data $u_0(x)=M\delta_0$ ($d=2$) or  $u_0(x)=M|x|^{-2}$ ($d\ge3$) and $\varphi_0=0$, there exists $\tau(M)$ such that for $\tau>\tau(M)$ the Cauchy problem has a global-in-time solution which is positive and self-similar. However, this solution may be nonunique for large $\tau$. 
\end{corollary}

When $d=2$, our uniqueness class should be compared with the nonuniqueness result obtained in \cite{BCD}, for self-similar solutions $u(x,t)=\frac1t U(x/\sqrt t)$ with large mass (depending on $\tau$) and with profiles $U\in {\mathcal C}_0(\R^2)$.

Corollary~\ref{sss} is a higher dimensional counterpart of the earlier result in \cite{BCD} on the self-similar solutions in $d=2$. And this can be interpreted as follows: {\em the dissipation in the Cauchy problem for the system {\rm (PP)} permits to define local- (and even global-in-time) solutions whenever $\tau$ is large enough} --- a striking difference compared to the solvability properties of the Cauchy problem for the system {\rm (PE)}, see e.g. \cite{BKP,BKWa}.  
%

\begin{proof}[Proof of Theorem~\ref{th:PM}]

Taking the Fourier transform in the second equation of {\rm (PP)}, we get
\begin{equation}
\label{eq:wiff'}
\widehat\varphi(\xi,t)={\rm e}^{-\tau^{-1}t|\xi|^2}\widehat\varphi_0(\xi)
+\tau^{-1} \int_0^t {\rm e}^{-\tau^{-1}(t-s)|\xi|^2}\widehat
u(\xi,s)\dd s.
\end{equation}
From the first equation, we get
\begin{equation}\label{F}
\begin{split}
\widehat u(\xi,t)
&={\rm e}^{-t|\xi|^2}\widehat u_0(\xi)
-\int_0^t i\xi {\rm e}^{-(t-s)|\xi|^2}\widehat{(u\nabla\varphi)}(\xi,s)\dd s
\\
&={\rm e}^{-t|\xi|^2}\widehat u_0(\xi)
  +(2\pi)^{-d}\int_0^t\!\!\int_{\R^d}\xi\eta {\rm e}^{-(t-s)|\xi|^2}\widehat{u}(\xi-\eta,s)\widehat\varphi(\eta,s)\dd\eta\dd s.\\
\end{split}
\end{equation} 
This leads us to introduce the linear operator $L$, depending on $\tau>0$, defined by
\begin{equation*}
\widehat{Lu}(\xi,t)=(2\pi)^{-d}\int_0^t \!\!\int_{\R^d} e^{-(t-s)|\xi|^2}\xi\eta\,\widehat u(\xi-\eta,s)
e^{-\tau^{-1}s|\eta|^2}\widehat\varphi_0(\eta)\dd \eta\dd s.
\end{equation*}
Then we introduce the bilinear operator, also dependent on~$\tau>0$,
\begin{equation}
\label{eq:bilf}
\widehat{B(u,v)}(\xi,t)=
(2\pi)^{-d}\int_0^t\!\!\int_0^s\!\!\int_{\R^d} 
\frac{\xi\eta}{\tau} 
{\rm e}^{-(t-s)|\xi|^2}{\rm e}^{-\frac{1}{\tau}(s-\sigma)|\eta|^2}
\widehat u(\xi-\eta,s)\widehat v(\eta,\sigma)\dd\eta\dd \sigma\dd s.
\end{equation}  
In this way, we see that $u$ satisfies the integral equation
\begin{equation}
\label{Fsol'}
\widehat u(\xi,t)=e^{-t|\xi|^2}\widehat u_0(\xi)+\widehat{Lu}(\xi,t)+\widehat{B(u,u)}(\xi,t).
\end{equation}

Eq.~\eqref{Fsol'} can be written as
\begin{equation}
\label{milde}
u=U_0+Lu+B(u,u),
\end{equation}
with $U_0(t)=e^{t\Delta}u_0$.
The standard fixed point lemma, in the formulation of, e.g., \cite[Theorem~3.1]{Rac}), applies to such equations.
This lemma asserts that the equation above has a unique solution~$u$ 
in the open ball 
\[
\{v\in Y\colon \|v\|_Y<r\},
\]
where $Y$ is a suitable Banach space. 
For this conclusion to be valid, three conditions have to be checked.
The first one is  continuity of the bilinear operator $B\colon Y\times Y\to Y$.
The second condition is the continuity of the linear one, $L\colon Y\to Y$, with 
$\|L\|_{\mathscr{L}(Y)}<1$.
The third condition is that $U_0\in Y$ should be of small enough norm:
for example, when
\begin{subequations} 
\begin{equation}
\label{siL}
\|L\|_{\mathscr{L}(Y)}\le \textstyle\frac12,
\end{equation}
it is sufficient to assume that 
\begin{equation}
\label{size}
\|U_0\|_Y<1/(16|\!|\!|B|\!|\!|).
\end{equation}
\end{subequations}
In this case, the existence and the uniqueness of the solution hold in the open ball
of radius $r=1/(4|\!|\!|B|\!|\!|)$.
See \cite{Rac} for more details.

We now establish the relevant estimates in the following lemmas.

\begin{lemma}
\label{lem:heatpm}
Let $d\ge2$, $a\ge d-2$ and $u_0\in \PM^{d-2}$, $U_0(t)=e^{t\Delta}u_0$. Then
$U_0\in\Y_a$ and there exists $c=c(a,d)>0$ independent of $u_0$ such that
\[
\|U_0\|_{\Y_a}\le c\,\|u_0\|_{\PM^{d-2}}.
\]  
\end{lemma}

\begin{proof}
Indeed, for any $t>0$ and $\xi\in\R^d$,
\[
\begin{split}
t^{1+(a-d)/2}|\xi|^a {\rm e}^{-t|\xi|^2}|\widehat u_0(\xi)|
&\le t^{1+(a-d)/2}|\xi|^{a-d+2}{\rm e}^{-t|\xi|^2}\|u_0\|_{\PM^{d-2}}\\
&\le c(a,d)\|u_0\|_{\PM^{d-2}},
\end{split}
\]
where $c(a,d)=\sup_{\rho>0}\rho^{1+(a-d)/2}{\rm e}^{-\rho}=(1+(a-d)/2)^{1+(a-d)/2}{\rm e}^{-1-(a-d)/2}$.
\end{proof}

The next two lemmas will be useful for the bilinear estimate of $B$ and the
linear one of $L$.
\begin{lemma}
\label{lem:convo}
Let $0<\alpha,\beta<d$ such that $\alpha+\beta>d$.
Then
\[
|x|^{-\alpha}*|x|^{-\beta}=C(\alpha,\beta,d)|x|^{-(\alpha+\beta)+d},
\]
with
\begin{equation}
\label{eq:CC}
C(\alpha,\beta,d)=\pi^{d/2}
\frac{\Gamma(\frac{d-\alpha}{2})\Gamma(\frac{d-\beta}{2})\Gamma(\frac{\alpha+\beta-d}{2})}{\Gamma(\frac{\alpha}{2})\Gamma(\frac{\beta}{2})\Gamma({d-\frac{\alpha+\beta}{2})}}.
\end{equation}
\end{lemma}
\begin{proof}
See \cite[Lemma 2.1]{BCGK}.
\end{proof}
By the properties of the Euler Gamma function $\Gamma$  we have, for $0<\alpha,\beta<d$ such that $\alpha+\beta>d$,
\begin{equation}
\label{eq:Ga}
C(\alpha,\beta,d)\lesssim \frac{\alpha\beta\,(2d-\alpha-\beta)}{(d-\alpha)(d-\beta)(\alpha+\beta-d)}.
\end{equation}

The following lemma is a slightly refined version of \cite[Lemma 3.2]{Rac}.
\begin{lemma}
\label{lem:integ}
Let $s>0$, $A>0$, $\delta>0$, $0\le b\le1$ and $\delta_*=\delta\wedge 1$.
Then
\[
\int_0^s {\rm e}^{-(s-\sigma)A}\sigma^{-1+\delta}\dd \sigma\le 4\delta_*^{-1}A^{-b}s^{\delta-b}.
\]
\end{lemma}

\begin{proof}
Indeed, for $0<\delta\le1$, splitting the integral at $\sigma=s/2$, we have
\[
\begin{split}
\int_0^s {\rm e}^{-(s-\sigma)A}\sigma^{-1+\delta}\dd \sigma
&\le \delta^{-1}{\rm e}^{-sA/2} (s/2)^{\delta} + (s/2)^{-1+\delta}\int_{s/2}^{s} {\rm e}^{-(s-\sigma)A}\dd \sigma\\
&\le(s/2)^\delta\Bigl(\delta^{-1}{\rm e}^{-sA/2}+(As/2)^{-1}\int_0^{As/2} {\rm e}^{-\sigma'}\dd \sigma'\Bigr)\\
&\le s^\delta\Bigl(\delta^{-1}{\rm e}^{-sA/2}+(1\wedge \textstyle\frac{2}{As})\Bigr)\\
&\le s^\delta\Bigl((\delta^{-1}+1)(1\wedge \textstyle\frac{2}{As})\Bigr)\\
&\le\frac{4 s^\delta}{\delta}(1\wedge \textstyle\frac{1}{As})\\
&\le \frac{4 s^\delta}{\delta}\Bigl(\frac{1}{As}\Bigr)^{b}
\end{split}
\]
for all $0\le b\le 1$.
In a similar way, for $\delta\ge1$ one easily gets that, for any $0\le b\le 1$.
\[
\int_0^s {\rm e}^{-(s-\sigma)A}\sigma^{-1+\delta}\dd \sigma\le 
 4 s^\delta\Bigl(\frac{1}{As}\Bigr)^{b}.
\]
\end{proof}

Let us now establish the relevant bilinear estimates. 
\begin{lemma}
\label{lem:d7}
Let $\tau>0$.
For $d=2$, the bilinear operator $B\colon\Y_a\times \Y_a\to \Y_a$ is continuous for $\frac12<a<2$.
For $d\ge3$, the bilinear operator $B\colon\Y_a\times \Y_a\to \Y_a$ is continuous for $d-2\le a<d$ and $a\not=1$. Moreover, when $a$ belongs to these ranges, the bilinear operator $B\colon\Y_a\times \Y_a\to \mathcal{X}$
 is also continuous.
\end{lemma}

\begin{remark}
 For $d=2$, A. Raczyński~\cite{Rac} proved a slightly different result, i.e., the continuity of
 $B\colon(\mathcal{X}\cap\Y_a)\times(\mathcal{X}\cap\Y_a)\to(\mathcal{X}\cap\Y_a)$.
 His estimates are uniform with respect to $\tau$.
\end{remark}

In fact, we will establish a more precise version of Lemma~\ref{lem:d7}, in order to take advantage of the fact that the norm of the bilinear operator $B$ in $\Y_a$ does depend on $\tau$, namely 
\begin{lemma}
\label{lem:bilinear}
Let $\tau>0$.
If $a$ and $b$ are such that 
\begin{subequations}
\begin{equation}
\label{cond-ab}
 \begin{cases}
 d-2b\le a<d-b\\
 0< b\le 1, \;a\not=1
 \end{cases}
 \qquad\text{when $d\ge3$,}
 \end{equation}
or
\begin{equation}
\label{cond-ab2}
\begin{cases}
 \frac32-b<a<2-b\\
 2-2b\le a\\
 0<b\le 1
\end{cases}
\qquad\text{when $d=2$}
\end{equation}
\end{subequations}
(when  $a$ satisfies the conditions of Lemma \ref{lem:d7}, then it is always possible to find $b$ satisfying~\eqref{cond-ab}--\eqref{cond-ab2}), then the following bilinear estimates hold 
\begin{subequations}
\begin{equation}
 \label{eq:bilest}
 \|B(u,v)\|_{\Y_a}\le K\,\tau^{b-1}\|u\|_{\Y_a}\|v\|_{\Y_a},
\end{equation}
and
\begin{equation}
 \label{eq:bilest2}
 \|B(u,v)\|_{\mathcal{X}}\le K\,\tau^{b-1}\|u\|_{\Y_a}\|v\|_{\Y_a},
\end{equation}
\end{subequations}
where $K=K(a,b,d)$ is independent of $\tau$, $u$ and $v$.
\end{lemma}

\begin{proof} 
Without loss of generality we can take $u$ and $v$ in $\Y_a$ with $\|u\|_{\Y_a}=\|v\|_{\Y_a}=1$. Then we have the estimate
\begin{equation*}
\begin{split}
 |\widehat B(&u,v)(\xi,t)|\\
 &\le
 (2\pi)^{-d}\int_0^t\!\!\int_0^s\!\!\int_{\R^d} 
\frac{|\xi|}{\tau} 
{\rm e}^{-(t-s)|\xi|^2}{\rm e}^{-\frac{1}{\tau}(s-\sigma)|\eta|^2} s^{-1+(d-a)/2}\sigma^{-1+(d-a)/2}|\xi-\eta|^{-a}|\eta|^{-a+1}
\dd\eta\dd \sigma\dd s.
\end{split}
 \end{equation*}
Our conditions imply
\[
a<d.
\] 
Applying Lemma~\ref{lem:integ} with $A=|\eta|^2/\tau$ and $\delta=(d-a)/2$ in the first inequality, $(d-a)_*=\min(d-a,1)$, and then Lemma~\ref{lem:convo} in the second inequality,
we get
\begin{equation*}
\begin{split}
 |\widehat B(&u,v)(\xi,t)|\\
 &\le
 {\textstyle\frac{8}{(2\pi)^d(d-a)_*}} \,\tau^{b-1}
 \int_0^t\!\!\int_{\R^d} 
|\xi| 
{\rm e}^{-(t-s)|\xi|^2} s^{-1+d-a-b}|\xi-\eta|^{-a}|\eta|^{-a+1-2b}
\dd\eta\dd s\\
&\le 
{\textstyle\frac{8\, C(a,a-1+2b,d) }{(2\pi)^d(d-a)_*}} \,\tau^{b-1}\,  |\xi|^{-2a-2b+2+d}
 \int_0^t {\rm e}^{-(t-s)|\xi|^2} s^{-1+d-a-b}\dd s.
\end{split}
\end{equation*}
Let us now apply Lemma~\ref{lem:integ} with $A=|\xi|^2$ and $\delta=d-a-b$ in the line above.
We obtain, for any $0\le \gamma\le 1$:
\begin{equation}
\label{eq:gab}
 |\widehat B(u,v)(\xi,t)|
  \le 
  K \,\tau^{b-1}\, t^{d-a-b-\gamma}|\xi|^{-2a-2b-2\gamma+d+2},
\end{equation}
where
\begin{equation}
 \label{eq:KK}
 K=\textstyle\frac{32\, C(a,a-1+2b,d)}{(2\pi)^d(d-a)_*(d-a-b)_*}.
\end{equation}
Let us first prove the bilinear estimate~\eqref{eq:bilest}.
In~\eqref{eq:gab}, in order to get $B(u,v)\in \Y_a$, we want
\begin{equation*}
\label{eq:syst}
 \begin{cases}
 -2a-2b-2\gamma+d+2=-a\\
 d-a-b-\gamma=-1-(a-d)/2,
 \end{cases}
 \end{equation*}
 and this system is equivalent to the single equation 
\begin{equation*}
 \gamma=-b+1+(d-a)/2.
\end{equation*}
In the above, the first application of Lemma~\ref{lem:integ} required $a<d$ and $0\le b\le 1$. The application of Lemma~\ref{lem:convo}
required $0<a<d$, $0<a-1+2b<d$ and $2a+2b-1>d$. The second application of Lemma~\ref{lem:integ} required $d-a-b>0$ and $0\le \gamma\le1$.
All these conditions are satisfied when $a$ and $b$ verify the assumptions of Lemma \ref{lem:bilinear}.
This proves~\eqref{eq:bilest}.

\bigskip

Let us now prove the second bilinear estimate~\eqref{eq:bilest2}.
For this, we have to take a different choice for $\gamma$, namely $\gamma=d-a-b$. Indeed, this implies
\begin{equation*}
\label{eq:syst2}
 \begin{cases}
 -2a-2b-2\gamma+d+2=-d+2\\
 d-a-b-\gamma=0,
 \end{cases}
 \end{equation*}
so that $B(u,v)\in\mathcal{X}$ by~\eqref{eq:gab}.

For this choice of $\gamma$ to be admissible, we need to put the new restriction 
$0\le d-a-b\le1$. As before, we need also $0<a<d$, $0<a-1+2b<d$, $2a+2b-1>d$ and $d-a-b>0$ and $0\le b\le1$. All this conditions are satisfied under the assumptions of the lemma.
\end{proof}

Asymptotically with respect to $\tau$, the most interesting choices for the parameter $b$ are:  
\begin{itemize}
\item[i)] $b\searrow 0$ (and $a\nearrow d$) when $\tau >\!\!\!>1$, or otherwise
\item[ii)] $b=1$ when $0<\tau<\!\!\!<1$.
\end{itemize}

Let us study the behaviour of the constant 
\begin{equation}
 K(a,b,d)=\frac{32\, C(a,a-1+2b,d)}{(2\pi)^d(d-a)_*(d-a-b)_*}
\end{equation}
appearing in Lemma \ref{lem:bilinear}.
When $b\searrow 0$ and $a\nearrow d$, recalling formula \eqref{eq:CC}
we see that
\[
C(a,a-1+2b,d)\approx\Gamma\Bigl(\frac{d-a}{2}\Bigr)\approx\frac{1}{d-a},
\]
and thus
\[
 K(a,b,d)\approx \frac{1}{(d-a)^{2}(d-a-b)}.
\]
In each dimension $d\ge2$ and for each $0<b\le 1$, we can always choose in Lemma~\ref{lem:bilinear}, for example,
 $a=d-\frac43b$. Then\footnote{
When $d\ge4$, instead of taking $a=d-\frac43b$, one can also take $a=d-2b$ for any 
$0<b\le1$, and reproduce the same calculations as done here.
This choice is perhaps more natural, at least when $b=1$,
because it allows to construct a unique solution directly in a ball of 
$\mathcal{X}=\Y_{d-2}$,
under the appropriate size conditions on the data.
} we get
\[
K(d-\textstyle\frac43b,b,d)\approx b^{-3},\qquad \text{as $b\searrow 0$}.
\]
But $K(d-\textstyle\frac43b,b,d)$ remains bounded when $0<b\le 1$ and $b$ is away from a neighborhood of $0$.
Therefore, we get from estimate~\eqref{eq:bilest},
\begin{equation}
\label{eq:bilp}
\|B(u,v)\|_{\Y_{d-\frac43b}}\le \frac{1}{16\,\kappa_d}\, b^{-3}\tau^{b-1} \|u\|_{\Y_{d-\frac43b}}\|v\|_{\Y_{d-\frac43b}}
\qquad(0<b\le 1,\quad d\ge2),
\end{equation}
for some constant $\kappa_d$ depending only on the dimension $d$.

On the other hand, the constant $c(a,d)$ defined in Lemma~\ref{lem:heatpm}
satisfies $c(a,d)=\psi(1+(a-d)/2)$, with the function $\psi(\rho)=\rho^\rho {\rm e}^{-\rho}$.
But $\psi(\rho)\le 1$ for any $0< \rho\le 1$, hence $c(a,d)\le 1$ for any $d-2\le a<d$.
In particular, recalling that $U_0(t)={\rm e}^{t\Delta}u_0$,
\begin{equation}
\label{dat}
\|U_0\|_{\Y_{d-\frac43b}} \le \|u_0\|_{\PM^{d-2}} \qquad (0<b\le 1,\quad d\ge2).
\end{equation}

Following the fixed point strategy, we now establish the relevant 
estimates for $\|L\|_{\mathscr{L}(\Y_{d-\frac43b})}$. 
This is the purpose of the next lemma.

\begin{lemma}
\label{lem:d=3}
Let $d\ge2$, $\varphi_0\in\mathscr{S}'(\R^d)$, with $\nabla\varphi_0\in\PM^{d-1}$.
Let $1<a<d$. 
Then $L\colon \Y_a\to \Y_a$ is continuous and
\begin{equation}
\label{eq:led1}
\|L\|_{\mathscr{L}(\Y_a)}\lesssim
\frac{\|\nabla\varphi_0\|_{\PM^{d-1}}}{(a-1)(d-a)^2}.
\end{equation}
Moreover,
if $d\ge3$ and $1<a<d-1$, then we have also
\begin{equation}
\label{eq:led2}
\|L\|_{\mathscr{L}(\Y_a)}\lesssim
\frac{\sqrt\tau\,\|\nabla\varphi_0\|_{\PM^{d-1}}}{(a-1)(d-a-1)}.
\end{equation}
For $d\ge2$ and $1<a<d$, if in addition $a\ge d-2$, then $L\colon\Y_a\to\mathcal{X}$ is continuous and 
\begin{equation}
\label{eq:led3}
\|L\|_{\mathscr{L}(\Y_a,\mathcal{X})}\lesssim
\frac{\|\nabla\varphi_0\|_{\PM^{d-1}}}{(a-1)(d-a)^2}.
\end{equation}
\end{lemma}

\begin{proof}

Let $u\in\Y_a$ such that $\|u\|_{\Y_a}=1$.
From the definition of $L$ we see that
\begin{equation*}
\begin{split}
|\widehat{Lu}(\xi,t)|
&\le
\|\nabla\varphi_0\|_{\PM^{d-1}}\int_0^t\!\!\int e^{-(t-s)|\xi|^2}|\xi|\,|\eta|^{-d+1}
|\xi-\eta|^{-a}e^{-\tau^{-1}s|\eta|^2} s^{-1+(d-a)/2}\dd\eta\dd s\\
&\le
\|\nabla\varphi_0\|_{\PM^{d-1}}\,|\xi|\int_0^t\!\!\int e^{-(t-s)|\xi|^2}|\xi-\eta|^{-a}|\eta|^{-d+1}\Bigl(1\wedge\frac{\tau}{s|\eta|^2}\Bigr) s^{-1+(d-a)/2}\dd s\dd \eta\\
&\le
\|\nabla\varphi_0\|_{\PM^{d-1}}\, |\xi|\int_0^t e^{-(t-s)|\xi|^2} I(\xi,\tau/s) s^{-1+(d-a)/2}\dd s.
\end{split}
\end{equation*}
Here,
\[
I(\xi,\tau/s)\equiv I_1+I_2,
\]
with 
\[
I_1=\int_{|\eta|\le \sqrt{\tau/s}}
  |\xi-\eta|^{-a}|\eta|^{-d+1}\dd \eta,
  \qquad\text{and}
  \qquad
I_2=\frac{\tau}{s}\int_{|\eta|\ge \sqrt{\tau/s}}
  |\xi-\eta|^{-a}|\eta|^{-d-1}\dd \eta.
  \]
The integral $I_1$ can be estimated splitting it into three terms: 
\[
I_1=
\Bigl(\int_{|\eta|\le \sqrt{\tau/s},\; |\eta|\le |\xi|/2}
+\int_{|\eta|\le \sqrt{\tau/s},\; |\xi-\eta|\le |\xi|/2}
+\int_{|\eta|\le \sqrt{\tau/s},\; |\eta|\ge |\xi|/2,\; |\xi-\eta|\ge |\xi|/2}
\Bigr)
|\xi-\eta|^{-a}|\eta|^{-d+1}\dd \eta.
\]
We find in this way,
\[
I_1
\lesssim |\xi|^{-a}\Bigl(|\xi|\wedge\sqrt{\frac{\tau}{s}}\Bigr)+
|\xi|^{-d+1}\int_{|\eta|\le \sqrt{\tau/s},\;|\xi-\eta|\le |\xi|/2} 
|\xi-\eta|^{-a}\dd \eta
+ \int_{|\xi|/2\le |\eta|\le \sqrt{\tau/s}}|\eta|^{-a-d+1}\dd\eta.
\]
Now, if $2\sqrt{\tau/s}\le|\xi|$, then the two last terms on the right-hand side are equal to zero, and we get
$I_1\lesssim |\xi|^{-a}\sqrt{\tau/s}$. 
Otherwise, if $|\xi|\le 2\sqrt{\tau/s}$,
then $I_1\lesssim (1+\frac{1}{d-a}+\frac{1}{a-1})|\xi|^{-a+1}$.
Hence, in any case,
\[
I_1\lesssim
\frac{|\xi|^{-a}}{(d-a)(a-1)}
\Bigl(\sqrt{\frac{\tau}{s}}\wedge|\xi|\Bigr).
\]
To estimate $I_2$, we split the region $\{\eta\colon|\eta|\ge\sqrt{\tau/s}\}$
into three regions just as before, and we easily get:
\[
\begin{split}
I_2
&\lesssim
\frac{\tau}{s}\biggl(
 |\xi|^{-a}\int_{\sqrt{\tau/s}\le|\eta|\le|\xi|/2} |\eta|^{-d-1}\dd\eta
 +|\xi|^{-d-1}\int_{|\eta|\ge \sqrt{\tau/s},\;|\xi-\eta|\le |\xi|/2} 
|\xi-\eta|^{-a}\dd \eta\\
&\qquad\qquad\qquad
	+ \int_{|\eta|\ge\max(|\xi|/2,\sqrt{\tau/s})}|\eta|^{-a-d-1}\dd\eta
\biggr).
\end{split}
\]
If $\frac32|\xi|\le\sqrt{\tau/s}$, then the first two terms in the right-hand side vanish and therefore
$I_2\lesssim (\sqrt{\tau/s})^{-a+1}$.
Otherwise, if $\sqrt{\tau/s}\le\frac32|\xi|$, then we get
$I_2\lesssim |\xi|^{-a}\sqrt{\tau/s} +(\frac1{d-a}+1)\frac{\tau}{s}|\xi|^{-a-1})
\lesssim \frac1{d-a}|\xi|^{-a}\sqrt{\tau/s}$.

Combining the two estimates for $I_1$ and $I_2$ we deduce
\[
I(\xi,\tau/s)=I_1+I_2
\lesssim \frac{|\xi|^{-a}}{(d-a)(a-1)}\Bigl(\sqrt{\frac\tau{s}}\wedge |\xi|\Bigr).
\]
Therefore, going back to the estimate for $\widehat L$, we deduce
\begin{equation}
\label{eq:poL}
|\widehat{Lu}(\xi,t)|
\lesssim
\frac{\|\nabla\varphi_0\|_{\PM^{d-1}}|\xi|^{1-a}}{(d-a)(a-1)}\int_0^t e^{-(t-s)|\xi|^2}\Bigl(\sqrt{\frac\tau{s}}\wedge |\xi|\Bigr)
s^{-1+(d-a)/2}\dd s.
\end{equation}
Applying Lemma~\ref{lem:integ}, with $A=|\xi|^2$ and $b=1$ we deduce
\[
|\widehat{Lu}(\xi,t)|
\lesssim
\|\nabla\varphi_0\|_{\PM^{d-1}}\frac{|\xi|^{-a}t^{-1+(d-a)/2}}{(d-a)^2(a-1)}
\qquad (1<a<d).
\]
Estimate~\eqref{eq:led1} follows.
If, instead, we apply Lemma~\ref{lem:integ}, 
with $A=|\xi|^2$ and $b=\frac12$ we deduce
\[
|\widehat{Lu}(\xi,t)|
\lesssim
\|\nabla\varphi_0\|_{\PM^{d-1}}\frac{\sqrt{\tau}\,|\xi|^{-a}t^{-1+(d-a)/2}}{(d-a-1)(a-1)}
\qquad(1<a<d-1,\quad d\ge3).
\]
This implies estimate~\eqref{eq:led2}.

In order to establish estimate~\eqref{eq:led3} we go back to inequality
~\eqref{eq:poL} and apply Lemma~\ref{lem:integ} with $A=|\xi|^2$ and $b=(d-a)/2$.
This implies, for $d-2\le a<d$,
\[
|\widehat{Lu}(\xi,t)|
\lesssim
\|\nabla\varphi_0\|_{\PM^{d-1}}\frac{|\xi|^{2-d}}{(d-a)^2(a-1)}.
\]
The claimed estimate follows.
\end{proof}

There is an alternative approach for the estimate of $L$.
Instead of considering the integral in the $\eta$-variable before that in the~$s$-variable, as we did at the beginning of the previous proof, we could inverse the order of two estimates.
It turns out that this alternative approach gives a slightly worse result
when $d\ge3$, but it has the advantage that it goes through also for $d=2$.
Namely let us prove the following:
\begin{equation}
\label{eq:led4}
\begin{split}
\|L\|_{\mathscr{L}(\Y_{d-\frac43})}
&\lesssim \sqrt \tau|\ln\textstyle\frac{\tau}{{\rm e}}| \,\|\nabla\varphi_0\|_{\PM^{d-1}}\\
\|L\|_{\mathscr{L}(\Y_{d-\frac43},\mathcal{X})}
&\lesssim \sqrt \tau|\ln\textstyle\frac{\tau}{{\rm e}}| \,\|\nabla\varphi_0\|_{\PM^{d-1}}
\end{split}
\qquad
(d\ge2,\quad 0<\tau\le 1).
\end{equation}
(For $d\ge3$, estimates \eqref{eq:led2}-\eqref{eq:led3} are available and give a better result).

Here is how we can proceed.
Going back to the expression of $\widehat L(\xi,t)$, we have, for any $\beta\ge0$,
and a suitable constant $C_\beta>0$,
\begin{equation}
\label{altest}
\begin{split}
|\widehat{Lu}(\xi,t)|
&\lesssim
\|\nabla \varphi_0\|_{\PM^{d-1}}\int_0^t\!\!\int e^{-(t-s)|\xi|^2}|\xi|\,|\eta|^{-d+1}
|\xi-\eta|^{-a}e^{-\tau^{-1}s|\eta|^2} s^{-1+(d-a)/2}\dd\eta\dd s\\
&\lesssim C_\beta\,\tau^\beta
\|\nabla \varphi_0\|_{\PM^{d-1}}
\int |\xi|\,|\eta|^{-d+1-2\beta}
|\xi-\eta|^{-a} \int_0^t e^{-(t-s)|\xi|^2} s^{-1-\beta+(d-a)/2}\dd s\dd\eta.\\
\end{split}
\end{equation}

Now, when $0\le \beta<(d-a)/2$, we can drop the constant $C_\beta$ and Lemma~\ref{lem:integ} applies (with $A=|\xi|^2$ and $b=1-\beta$) yielding
\begin{equation}
\label{eq:fin}
\begin{split}
|\widehat{Lu}(\xi,t)|
&\lesssim
\frac{\tau^\beta}{d-a-2\beta}\, \|\nabla \varphi_0\|_{\PM^{d-1}}
 |\xi|^{1-2(1-\beta)} t^{-1+(d-a)/2} \int |\xi-\eta|^{-a}|\eta|^{-d+1-2\beta}\dd\eta.
\end{split}
\end{equation}
Therefore, for 
$0<a<d$, $0\le \beta<\frac12$ and $1< a+2\beta<d$, we get by Lemma~\ref{lem:convo},
\begin{equation}
\label{eq:esL0}
\begin{split}
\|L\|_{\mathscr{L}(\Y_a)}
&\lesssim \frac{\tau^\beta\,C(a,d-1+2\beta,d)}{d-a-2\beta} \|\nabla \varphi_0\|_{\PM^{d-1}}\\
&\lesssim \frac{\tau^\beta \,a\,(d-1+2\beta)(d-a+1-2\beta)}{(d-a)(1-2\beta)(a+2\beta-1)(d-a-2\beta)} \|\nabla \varphi_0\|_{\PM^{d-1}},
\end{split}
\end{equation}
where the function $C$ is given by~\eqref{eq:CC},
the latter inequality being a consequence of the bound~\eqref{eq:Ga}.

We can take $a=d-\frac43$ and $\frac14\le \beta< \frac12$ (so, in particular, three of the four factors in the denominator are bounded away from zero).
Then we get
\begin{equation}
\label{eq:esL2}
\begin{split}
\|L\|_{\mathscr{L}(\Y_{d-\frac43})}
&\lesssim \|\nabla \varphi_0\|_{\PM^{d-1}}\frac{\tau^\beta}{1-2\beta}\qquad(\textstyle\text{for $\frac14\le \beta<\frac12$}).
\end{split}
\end{equation}
Minimizing the coefficient on the right-hand side with respect to $\beta$
leads to choosing $\beta=\frac12+\frac1{\ln\tau}$.
When $0<\tau\le{\rm e}^{-4}$, with this choice we do have 
$\beta\in[\frac14,\frac12)$ and so
\begin{equation}
\label{dor}
\|L\|_{\mathscr{L}(\Y_{d-\frac43})}\lesssim 
\|\nabla \varphi_0\|_{\PM^{d-1}}
\sqrt\tau|\ln \tau|  , \qquad
\qquad(0<\tau\le {\rm e}^{-4}).
\end{equation}
When ${\rm e}^{-4}\le \tau\le 1$ the choice of $\beta$ is not important ($\beta=\frac14$ will do).
Then, the first inequality of~\eqref{eq:led4} follows.

To establish the second estimate in~\eqref{eq:led4} we restrict ourselves, as before, to $0<a<d$, $0\le \beta<\frac12$ and  $1<a+2\beta<d$.
If, in addition, we assume $d-a-2\beta\le 2$, then we can apply, in~\eqref{altest},
Lemma~\ref{lem:integ} in a different way than before (now with $A=|\xi|^2$ and $b=\frac{d-a}{2}-\beta$), and we obtain
the following modification of estimate~\eqref{eq:fin}:
\[
\begin{split}
|\widehat{Lu}(\xi,t)|
&\lesssim 
\frac{\tau^\beta}{d-a-2\beta}\,\|\nabla \varphi_0\|_{\PM^{d-1}}|\xi|^{1-(d-a)+2\beta} \int|\xi-\eta|^{-a}|\eta|^{-d+1-2\beta}\dd \eta\\
&\lesssim
\frac{\tau^\beta C(a,d-1+2\beta,d)}{d-a-2\beta} |\xi|^{-d+2}\|\nabla \varphi_0\|_{\PM^{d-1}}.
\end{split}
\]
Therefore, under the previous conditions on the parameters
we obtain
\begin{equation}
\label{eq:Liii}
\|L\|_{\mathscr{L}(\Y_a,\mathcal{X})}
\lesssim \frac{\tau^\beta C(a,d-1+2\beta,d)}{d-a-2\beta}\|\nabla \varphi_0\|_{\PM^{d-1}}.
\end{equation}
We can proceed as before, choosing $a=d-\frac43$,
and $\beta=\frac12+\frac{1}{\ln\tau}\in [\frac14,\frac12)$
when $0<\tau\le {\rm e}^{-4}$, or $\beta=\frac14$ when 
${\rm e}^{-4}\le \tau\le 1$.
We readily get the second inequality in~\eqref{eq:led4}.

\medskip

We are now in a position to complete the proof of Theorem~\ref{th:PM}.

\paragraph{The case $d\ge3$.}
For $0<\tau\le 1$ we use the estimates~\eqref{eq:bilp}--\eqref{dat}
with $b=1$. We also apply estimate~\eqref{eq:led2}
with $a=d-\frac43$.
Thus,
\[
\|U_0\|_{\Y_{d-\frac43}}\le \|u_0\|_{\PM^{d-2}},
\qquad
\|B(u,v)\|_{\Y_{d-\frac43}}\lesssim 
\frac{1}{16\,\kappa_d}\|u\|_{\Y_{d-\frac43}}\|v\|_{\Y_{d-\frac43}}
\]
and, for some constant $\tilde\kappa_d>0$,
\[
\|L\|_{\mathscr{L}(\Y_{d-\frac43})}
\le \frac{\sqrt\tau\,\|\nabla\varphi_0\|_{\Y_{d-\frac43}}}{2\,\tilde\kappa_d}.
\]
We deduce that the fixed point lemma~\cite[Theorem~3.1]{Rac} applies in the ball of the space 
$\Y_{d-\frac43}$ with center $0$ and radius $4\kappa_d$, under the smallness conditions~\eqref{small-assu0}.

When $\tau\ge1$, we make use of estimates~\eqref{eq:bilp}-\eqref{dat} 
with a general $0<b\le1$, and of
estimate~\eqref{eq:led1} with $a=d-\frac43 b$.
The latter can be written as (after taking, if necessary, a larger value for
the constant $\tilde\kappa_d$)
\[
\|L\|_{\mathscr{L}(\Y_{d-\frac43b})}
\le \frac{\|\nabla\varphi_0\|_{\Y_{d-\frac43b}}}{2\,\tilde\kappa_d\, b^2}.
\]
Now, for any $0<b\le 1$, the fixed point lemma applies in the ball of the space 
$\Y_{d-\frac43b}$, with center $0$ and radius $4\kappa_d b^3\tau^{1-b}$, under the smallness condition~\eqref{small-assu}.

The solution constructed above, in the space $\Y_{d-\frac43}$ ($0<\tau\le1$),
or in the space $\Y_{d-\frac43b}$ ($\tau\ge1$) belongs in any case also to $\mathcal{X}$.
This is a consequence of the 
the fact that
\[
u=U_0+Lu+B(u,u),
\] 
and that each one of the terms in the right-hand side belongs to $\mathcal{X}$.
Indeed, to see this we just have to apply the last assertion of Lemma~\ref{lem:d7}, estimate~\eqref{eq:led3}, and the elementary fact that the heat semigroup 
$u_0\mapsto U_0$ is bounded from
$\PM^{d-2}$ to~$\mathcal{X}$ (see Lemma~\ref{lem:heatpm}, case $a=d=2$). 
This establishes the theorem in the case $d\ge3$.

\paragraph{The case $d=2$.}

When $0<\tau\le1$, we take as before $a=d-\frac43=\frac23$. 
Estimates~\eqref{eq:bilp}-\eqref{dat} are still available, but estimates~\eqref{eq:led1}-\eqref{eq:led2} are not, because $a<1$.
But the slightly more stringent smallness condition on $\varphi_0$ that we assumed
in the case $d=2$ allow us to apply the rougher estimate~\eqref{eq:led4}.
We can conclude as before that there exists a solution $u\in \Y_{2/3}\cap\mathcal{X}$,
which is unique in the ball of $\Y_{2/3}$ centered at the origin and with radius $4\kappa_2$.

When $\tau\ge1$ and $0<b\le\frac12$ (or $0<b\le 1$ when $\varphi_0=0$),
the proof that we presented in the case $d\ge3$ goes through also when $d=2$ without any change,
because $d-\frac43 b\ge \frac43>1$ (the latter condition is needed in the application of~\eqref{eq:led1} and \eqref{eq:led2}).
The case $\tau\ge1$ and $\frac12<b\le 1$ is not interesting, as discussed in Remark~\ref{rem:di}.
\end{proof}

\section{Study of {\rm(PP)} in Besov spaces}

The functional setting considered in the previous section is somewhat restrictive. In this section, we will study the problem {\rm (PP)} in 
Besov-type spaces,  generalizing the approach of \cite{BGK} to $d \ge 3$. 
For sake of simplicity we limit ourselves to the case $\varphi_0=0$.
Then the integral formulation~\eqref{milde} of {\rm (PP)} simplifies to
\begin{equation}
\label{milde}
u=U_0+B(u,u),
\qquad
\text{with}\quad U_0(t)=e^{t\Delta}u_0.
\end{equation}

Let
\begin{equation} \label{E_p}
{\mathcal E}_p:=\left\{ u \in L^{\infty}(0,\infty;L^p(\R^d)),\ \III{u}_p:=\esssup_{t>0} t^{1-d/(2p)} \Vert u \Vert_p <\infty\ \right\}.
\end{equation}
Note that this space is invariant with respect to the rescaling
\[
u_{\lambda}(x,t):=\lambda^2 u(\lambda x,\lambda^2 t).
\]
We recall the classical inequalities $L^p-L^q$ for the heat semigroup (see e.g. \cite[Sec. 15.1]{TayIII}), under the condition $1\le p \le q\le\infty$:
\begin{equation} \label{heat}
\Vert {\rm e}^{t\Delta} f \Vert_q \le C(d,p,q) t^{-d(1/p-1/q)/2} \Vert f \Vert_p,
\qquad
 \Vert \nabla {\rm e}^{t\Delta} f \Vert_q \le C(d,p,q) t^{-1/2-d(1/p-1/q)/2} \Vert f \Vert_p.
\end{equation}
Now we study the (respectively, linear and bilinear) operators $\mathbb{L}$ and $B$ given by:
\begin{align} \label{L}
&\mathbb{L}z(t):=\tau^{-1} \int_{0}^{t}{\nabla {\rm e}^{\tau^{-1} (t-s) \Delta} z(s) \ds},\ 
\\ \label{B}
&B(u,z)(t):=-\int_{0}^{t}{\nabla {\rm e}^{(t-s) \Delta} \cdot(u(s) \mathbb{L}z(s)) \ds}.
\end{align}

Below, all constants are implicitly assumed to depend on $d$. They are also implicitly assumed to depend on $p$ and $q$ in the lemmas; this dependence is only made explicit in the theorems.

\begin{lemma} \label{BesL1}
If $0 \le 1/q \le 1/p < 1/q+1/d$ and $0<1/p$, then
$$
\III{\mathbb{L}z(t)}_q \le C \tau^{-1/2+d/2(1/p-1/q)} \III{z}_p.
$$
\end{lemma}

\begin{proof} For $t>0$,
\begin{align*}
& \Vert \mathbb{L}z(t) \Vert_q \le C \tau^{-1} \int_{0}^{t}{ \Big(\tau^{-1}(t-s) \Big)^{-1/2-d/2(1/p-1/q)} \Vert z(s) \Vert_p \ds}
\\
& \le C \tau^{-1/2+d/2(1/p-1/q)} \int_{0}^{t}{ (t-s)^{-1/2-d/2(1/p-1/q)} s^{-1+d/2p} \III{z}_p \ds} 
\\
& \le C \tau^{-1/2+d/2(1/p-1/q)} t^{-1/2+d/2q} \III{z(t)}_p.
\end{align*}
\end{proof}

\begin{lemma} \label{BesL2}
If $1/d < 1/p+1/q \le 1$, $0 \le 1/q < 1/d$ and the conditions of Lemma~\ref{BesL1} hold, then we have
$$
\III{B(u,z)}_p \le C \tau^{-1/2+d/2(1/p-1/q)} \III{u}_p \III{z}_p.
$$
\end{lemma}

\begin{proof}
Consider $r$ such that $1/r=1/p+1/q$. Then we have, for $t>0$, using Lemma~\ref{BesL1}, 
\begin{align*}
& \Vert B(u,z)(t) \Vert_p \le C \int_{0}^{t}{(t-s)^{-1/2-d/2(1/r-1/p)} \Vert u(s) \mathbb{L}z(s) \Vert_r \ds}
\\
& \le C \int_{0}^{t}{(t-s)^{-1/2-d/2(1/r-1/p)} \Vert u(s) \Vert_p \Vert \mathbb{L}z(s) \Vert_q \ds}
\\
& \le C \Big( \int_{0}^{t}{(t-s)^{-1/2-d/2q} s^{-1+d/2p} s^{-1/2+d/2q}  \ds} \Big) \tau^{-1/2+d/2(1/p-1/q)} \III{u}_p \III{z}_p
\\
& \le  C t^{-1+d/2p} \tau^{-1/2+d/2(1/p-1/q)} \III{u}_p \III{z}_p.
\end{align*}
\end{proof}

Now let us check compatibility of the exponents in the two lemmas above. The assumptions above imply that:
$$
|1/p-1/d|<1/q \le \min(1/p,1-1/p);\ 1/q<1/d;\ 0<1/p;\ d>1.
$$
This requires choosing $p$ so that $1/d<2/p<\min(1+1/d,4/d)$.
\\
In other words, in order to find a suitable $q$ to use the lemmas above, a necessary and sufficient assumption is:
$$
d \ge 2;\ \max(d/2,\ 2d/(d+1))<p<2d.
$$ 

\begin{lemma} \label{BesL3}
Let $d \ge 2$, $p>d/2$. Then $u_0$ belongs to the Besov space 
$\dot{B}_{p,\infty}^{-(2-d/p)}(\R^d)$, if and only if ${\rm e}^{t\Delta}u_0 \in {\mathcal E}_p$. Moreover, there exists $C_p>0$ such that
\[
 C_p^{-1}\|u_0\|_{\dot{B}_{p,\infty}^{-(2-d/p)}} 
 \le 
 \III{{\rm e}^{t\Delta} u_0}_p \le C_p \Vert u_0 \Vert_{\dot{B}_{p,\infty}^{-(2-d/p)}}.
\]
\end{lemma}

\begin{proof}
See \cite[Theorem 2.34]{BCD-Book}.
\end{proof}

We deduce, using  the fixed point lemma in its standard form \cite[Lemma 13.2]{Lem-Book1},
the following theorem.

\begin{theorem} \label{BesT1}
Let $d \ge 2$, $\max(d/2,\ 2d/(d+1))<p<2d$, $u_0 \in \dot{B}_{p,\infty}^{-(2-d/p)}$ and $\varphi_0=0$.
Let $q$ such that
$$
|1/p-1/d|<1/q \le \min(1/p,1-1/p),\qquad  1/q<1/d.
$$
Then there exist constants $C_{p,q},\kappa_{p,q}>0$, independent of $\tau$ and $u_0$, such that if
$$
\Vert u_0 \Vert_{\dot{B}_{p,\infty}^{-(2-d/p)}} < C_{p,q} \tau^{1/2-d/2(1/p-1/q)},
$$ 
then {\rm (PP)} has a mild solution $u \in {\mathcal E}_p$, such that $\III{u}_p \le \kappa_{p,q} \Vert u_0 \Vert_{\dot{B}_{p,\infty}^{-(2-d/p)}}$. 
Moreover, if $u,\tilde{u}$ are two mild solutions of {\rm (PP)} satisfying 
$$
\III{u}_p \le R,\quad \III{\tilde{u}}_p \le R,\quad R<\kappa_{p,q} C_{p,q}\, \tau^{1/2-d/2(1/p-1/q)},
$$
then $u \equiv \tilde{u}$. 
\end{theorem}

For instance, applying Theorem \ref{BesT1} above to $p=q \in (d,2d)$ for $d \ge 2$, we obtain that there exists a constant $C_p>0$ such that for 
\begin{equation}
\Vert u_0 \Vert_{\dot{B}_{p,\infty}^{-(2-d/p)}} < C_p \sqrt{\tau},
\end{equation}
 {\rm (PP)} has a global solution $u \in {\mathcal E}_p$. 

Notice that, for $d\ge3$, Chandrasekhar data $u_0=A/|\cdot|^2$  
belong to $\dot{B}_{p,\infty}^{-(2-d/p)}$, for $p>d/2$. (This is an obvious consequence of any of the
two injections in Remark~\ref{rem:inj} below. Actually, only of one of them for $d=3$).
Thus, for $d \ge 3$ and $u_0(x)=A\sqrt{\tau}/|x|^2$, there exists a self-similar solution to {\rm (PP)} starting from $u_0$, provided $0 \le A < A_d$ for some $A_d$ depending only on $d$. 
Hence one can view Corollary~\ref{sss} both as a consequence of Theorem~\ref{th:PM} or of Theorem~\ref{BesT1}.

\begin{remark}
Theorem~\ref{BesT1} does not encompass the case $\tau=0$. Modifications are required.
The case $\tau=0$ was treated in \cite{I}.
\end{remark} 
\begin{remark} 
\label{rem:inj}
We have the continuous injection
\[
\PM^{d-2}\hookrightarrow \dot{B}_{p,\infty}^{-(2-d/p)},
\qquad \text{for $d\ge2$ and all $p\in(d/2,\infty]$ such that $p\ge2$}.
\]
Therefore, the above existence theorem is more general, in terms of regularity
of the initial data, than Theorem~\ref{th:PM}. On the other hand, for slightly more regular data, Theorem~\ref{th:PM} allows to take initial condition of much larger size for $u_0$ when $\tau\gg1$. The two theorems are thus complementary.
The above injection relies on the application of the Hausdorff--Young inequality.
Indeed, if $p\in [2,\infty]$, $p>d/2$ and $p'$ is the conjugate exponent of $p$, then we have, for some constant $c_p>0$ and all $\tau>0$, using the change of variables $\bar{\xi}=\sqrt{t} \xi$, 
\[
\begin{split}
\|e^{t\Delta}u_0\|_p 
&\le c_p\Bigl(\int e^{-tp'|\xi|^2}|\widehat u_0(\xi)|^{p'}\,\dd\xi\Bigr)^{1/p'}\\
&\le c_p\, t^{-1+d/(2p)}\|u_0\|_{\PM^{d-2}}\Bigl(\int e^{-p'|\bar{\xi}|^2}|\bar{\xi}|^{-(d-2)p'}\dd \bar{\xi} \Bigr)^{1/p'}.
\end{split}
\]
The injection now follows from Lemma~\ref{BesL3}.

Another relevant injection is
\[
L^{d/2,\infty}(\R^d) \hookrightarrow \dot{B}_{p,\infty}^{-(2-d/p)}(\R^d),
\qquad\text{for $d\ge3$ and all $p \in (d/2,\infty]$},
\]
where $L^{p,\infty}$ denotes the Lorentz space 
(see \cite[Chapter 2]{Lem-Book1}).
Indeed, if $f \in L^{d/2,\infty}(\R^d)$, then there exists $C_p>0$ such that $\Vert {\rm e}^{t\Delta} f \Vert_p \le C_p t^{-1+d/2p} \Vert f \Vert_{d/2,\infty}$.
\end{remark}

Now let us discuss the continuity of the mild solution with respect to time. In what follows, we assume that $u_0 \in M^{d/2} \cap \dot{B}_{p,\infty}^{-(2-d/p)}$, and $2d/3<p \le d$.
Notice that $M^{d/2} \not\hookrightarrow \dot{B}_{p,\infty}^{-(2-d/p)}$. 
A suitable example showing this is the derivative $\partial_{x_1}\psi$ with some $\psi\ge 0$ of compact support and not very smooth. 
 On the other hand, in the physically relevant case $u_0 \ge 0$, the condition $u_0 \in M^{d/2}$ is redundant because 
$
\bigl\{ u_0 \in \dot{B}_{d,\infty}^{-1},\ u_0 \ge 0 \bigr\} \subset M^{d/2}
$, 
see \cite[p. 1197]{Lem}. 

\begin{lemma} \label{BesL4}
If $d \ge 2$ and $u \in {\mathcal E}_p$, we have:
$$
\Vert B(u,u)(t) \Vert_{d/2} \le C \tau^{-3/2+d/p} \Vert u \Vert_p^2.
$$
\end{lemma}

\begin{proof}
First using H{\"o}lder's inequality and then Lemma~\ref{BesL1} for $2/d=1/p+1/q$, we get
\begin{align*}
& \Vert B(u,u)(t) \Vert_{d/2} \le C \int_{0}^{t}{(t-s)^{-1/2} \Vert u(s) \mathbb{L}u(s) \Vert_{d/2} \ds} \le C \int_{0}^{t}{(t-s)^{-1/2} s^{-1+d/2p} \Vert \mathbb{L}u(s) \Vert_{q} \ds} \cdot \III{u}_p
\\
& \le C \Big( \int_{0}^{t}{(t-s)^{-1/2} s^{-1+d/2p} s^{-1/2+d/2q} \ds} \Big) \tau^{-1/2+d/2(1/p-1/q)}\III{u}_p^2 =C \tau^{-3/2+d/p} \III{u}_p^2.
\end{align*}
\end{proof} 
In particular, if $u \in {\mathcal E}_p$, then $B(u,u) \in L^{\infty}(0,\infty;L^{d/2})$.
Before proving the two time-continuity lemmas, let us point out that there exists a constant $C_{\tau}>0$ such that we have
\begin{equation} \label{uLu1}
\Vert u(\sigma) \mathbb{L}u(\sigma) \Vert_{d/2} \le C_{\tau} \sigma^{-1/2}  \III{u}_p^2
\end{equation}
which follows from  the computation above.

\begin{lemma} \label{BesL5}
If $u \in {\mathcal E}_p$, then $B(u,u) \in {\mathcal C}((0,\infty),L^{d/2})$.
\end{lemma}

\begin{proof}
For $t \ge s >0$, $B(u,u)(t)-B(u,u)(s)=I_1+I_2$ with
$$
I_1=-\int_{0}^{s}{(\nabla {\rm e}^{(t-\sigma)\Delta}-\nabla {\rm e}^{(s-\sigma)\Delta})u(\sigma) \mathbb{L}u(\sigma) \dsigma},\ I_2=-\int_{s}^{t}{\nabla {\rm e}^{(t-\sigma) \Delta}(u(\sigma) \mathbb{L}u(\sigma)) \dsigma}.
$$
We have (denoting by $G$ the Gaussian kernel)
$$
\Vert I_1 \Vert_{d/2} \le \int_{0}^{s}{\underbrace{\Vert \nabla G(\cdot,t-\sigma) - \nabla G(\cdot,s-\sigma) \Vert_1}_{\rightarrow 0\text{ as } t \rightarrow s}  
\underbrace{\Vert u(\sigma) \mathbb{L}u(\sigma) \Vert_{d/2}}_{\le C_{\tau} \sigma^{-1/2} \text{ by } (\ref{uLu1})} \dsigma}
$$
Moreover,
$$
\Vert \nabla G(\cdot,t-\sigma) - \nabla G(\cdot,s-\sigma) \Vert_1 \le (t-\sigma)^{-1/2}+(s-\sigma)^{-1/2} \le 2(s-\sigma)^{-1/2},
$$
and
$$
\sigma \mapsto 2(s-\sigma)^{-1/2}\sigma^{-1/2} \in L^1(0,s).
$$
By the dominated convergence theorem, $\Vert I_1 \Vert_{d/2} \rightarrow 0$ as $t \searrow s$.
\\
On the other hand, for $s>0$, making the change of variables $\sigma=t\tilde{\sigma}$, we obtain that as $t \searrow s$,
\begin{align*}
\Vert I_2 \Vert_{d/2} &\le \Big( C \int_{s}^{t}{(t-\sigma)^{-1/2} \sigma^{-1/2}\dsigma} \Big) \tau^{-3/2+d/p} \III{u}_p^2 \\
&= \Big( C \int_{s/t}^{1}{(1-\tilde{\sigma})^{-1/2} \tilde{\sigma}^{-1/2}\ {\rm d}\tilde{\sigma}} \Big) \tau^{-3/2+d/p} \III{u}_p^2 \rightarrow 0.
\end{align*}
\end{proof}

\begin{lemma} \label{BesL6}
Let $u \in {\mathcal E}_p$. Then $B(u,u)(t) \rightarrow 0$ as $t \searrow 0$, in the sense of distributions, i.e. in the space $\mathscr D'$.
\end{lemma}

\begin{proof}
Consider a test function $\phi \in {\mathcal C}_0^{\infty}(\R^d)$. Then, for $t \searrow 0$, denoting $r$ the conjugate of $d/2$, i.e. $1/r+2/d=1$, by \eqref{uLu1} we get
\begin{align*}
& \Big| \int_{\R^d}{B(u,u)(x,t) \phi(x) \dx} \Big| = \Big| \int_{0}^{t} \int_{\R^d}{ {\rm e}^{(t-\sigma) \Delta} (u(\sigma) Lu(\sigma))(x) \nabla\phi(x) \dx \dsigma}  \Big|
\\
& \le C_{\tau} \int_{0}^{t}{\Vert u(\sigma) \mathbb{L}u(\sigma) \Vert_{d/2} \Vert \nabla \phi \Vert_{r}}\dsigma \le C_{\tau} \Big( \int_{0}^{t}{\sigma^{-1/2} \dsigma} \Big) \Vert \nabla \phi \Vert_{r} \rightarrow 0.
\end{align*}
\end{proof}

Now let $u_0 \in \dot{B}_{d,\infty}^{-1}$. Theorem~\ref{BesT1} applies with $p=d$ and any $q$ satisfying $2 \le d<q<\infty$. Therefore, for any such $q$ there exists $C_q>0$ such that if 
$\Vert u_0 \Vert_{\dot{B}_{d,\infty}^{-1}} \le C_q \tau^{d/2q}$, the solution $u \in {\mathcal E}_d$ built in this theorem satisfies
$$
u={\rm e}^{t\Delta}u_0+B(u,u),\ B(u,u) \in BC((0,\infty),L^{d/2});\ B(u,u)(t) \rightarrow 0\  \text{\ in\ } {\mathscr D}',\  {\rm as\ }t \searrow 0.
$$
Now assume, more precisely, that 
$$
u_0 \in M^{d/2} \cap {\dot{B}_{d,\infty}^{-1}},\ \Vert u_0 \Vert_{\dot{B}_{d,\infty}^{-1}} \le C_q \tau^{d/2q}.
$$
Then ${\rm e}^{\cdot\Delta}u_0 \in BC((0,\infty),M^{d/2})$  \cite{G-M}, and $B(u,u) \in BC((0,\infty),L^{d/2}) \subset BC((0,\infty),M^{d/2}) $. 
Moreover, as $t \searrow 0$, ${\rm e}^{t\Delta}u_0 \rightarrow u_0$ in ${\mathscr D'}$, and $B(u,u)(t) \rightarrow 0$ in ${\mathscr D'}$.

Summarizing the statements above, we get the following modification of Theorem~\ref{BesT1}:

\begin{theorem} \label{BesT2}
For any $q$ satisfying $2 \le d<q<\infty$ and any $\tau>0$, there exists $C_{q,d}>0$ (independent of $\tau$) such that if
$$
u_0 \in M^{d/2} \cap {\dot{B}_{d,\infty}^{-1}};\ \Vert u_0 \Vert_{\dot{B}_{d,\infty}^{-1}} \le C_{q,d} \tau^{d/2q},
$$
then there exists a mild solution $u \in {\mathcal E}_d \cap BC((0,\infty),M^{d/2})$ to {\rm (PP)} such that $u(t) \rightarrow u_0$ in ${\mathscr D'}$ as $t \searrow 0$ and 
$$
\III{u}_d \le 2 \Vert u_0 \Vert_{\dot{B}_{d,\infty}^{-1}} \le 2 C_{q,d} \tau^{d/2q}.
$$
A uniqueness class for mild solutions is the ball $\left\{ v \in {\mathcal E}_d:\ \III{v}_d < 2 C_{q,d} \tau^{d/2q} \right\}$.
\end{theorem}


\section*{Acknowledgments} 
The authors thank 
Lucilla Corrias, Pierre-Gilles Lemari\'e, Miko{\l}aj Sier\.z\c{e}ga and Philippe Souplet for interesting conversations on the topic of our work. 

The first named author would like to thank Institut Camille Jordan, Universit\'e Claude Bernard-Lyon~1 for hospitality during his sabbatical stay (Sep 2021--Jan 2022) as a fellow of {\em Institut des \'Etudes Avanc\'ees -- Collegium de Lyon}, partially supported by  the Polish NCN grant \hbox{2016/23/B/ST1/00434}.

\end{document}